\documentclass[12pt,reqno]{amsart}

\usepackage{amssymb, amsmath, amsfonts, latexsym}
\usepackage{enumerate}
\usepackage{color}
\newtheorem{thm}{Theorem}[]
\newtheorem{lem}{Lemma}[section]

\newtheorem{rmk}{Remark}[section]
\theoremstyle{definition}

\numberwithin{equation}{section} \theoremstyle{remark}

\title[$W_1$ between Chiral and Gumbel ]{\bf Optimal $W_1$ and Berry-Esseen bound between  the spectral radius of large Chiral non-Hermitian random matrices and Gumbel}

\author{Yutao Ma}
\address{Yutao MA\\ School of Mathematical Sciences $\&$ Laboratory  of Mathematics and Complex Systems of Ministry of Education, Beijing Normal University, 100875 Beijing, China.} 
\thanks{The research of Yutao Ma was supported in part by NSFC 12171038 and 985 Projects.}
\email{mayt@bnu.edu.cn}

\author{Siyu Wang}
\address{Siyu Wang\\ School of Mathematical Sciences $\&$ Laboratory of Mathematics and Complex Systems of Ministry of Education, Beijing Normal University, 100875 Beijing, China.}
\email{wang\_siyu@mail.bnu.edu.cn}

\begin{document}
\maketitle

\begin{abstract}

Consider the chiral non-Hermitian random matrix ensemble with parameters $n$ and $v$ and the non Hermiticity parameter $\tau=0$  and let $(\zeta_i)_{1\le i\le n}$ be its $n$ eigenvalues with positive $x$-coordinate. Set 
$$X_n:=\sqrt{\log s_n}\left(\frac{2n \max_{1\le i\le n}|\zeta_i|^2-2\sqrt{n(n+v)}}{\sqrt{2n+v}}-a(s_{n})\right)$$ with $s_n=n(n+v)/(2n+v)$ and $a(s_n)=\sqrt{\log s_n}-\frac{\log(\sqrt{2\pi}\log s_n)}{\sqrt{\log s_n}}.$ It was proved in \cite{JQ} that $X_n$ converges weakly to the Gumbel distribution $\Lambda$. In this paper, we give in further that  
$$\lim_{n\to\infty} \frac{\log s_n}{(\log\log s_n)^2}W_1\left(F_n, \Lambda\right)=\frac{1}{2}$$
and the Berry-Esseen bound 
$$\lim_{n\to\infty} \frac{\log s_n}{(\log\log s_n)^2}\sup_{x\in\mathbb{R}}|F_n(x)-e^{-e^{-x}}|=\frac{1}{2e}.$$
   Here, $F_n$ is the distribution (function) of $X_n.$    \end{abstract} 


{\bf AMS Classification Subjects 2020:} 60G70, 60B20, 60B10. 

\section{Introduction}\label{chap:intro}
For integers $n\ge 1$ and $v\ge 0,$ let $P$ and $Q$ be $(n+v)\times n$ matrices with i.i.d centered complex Gaussian entries of variance $1 /(4 \mathrm{n}).$ These are the building
blocks of the two correlated random matrices
$$
\Phi = \sqrt{1+\tau} P+\sqrt{1-\tau} Q
\text{\quad  and \quad }
\Psi = \sqrt{1+\tau} P-\sqrt{1-\tau} Q.
$$
Here, $\tau \in[0,1]$ is a non-Hermiticity parameter. Notably, for $\tau=0$, $\Phi$ and $\Psi$ are uncorrelated Gaussian random matrices, while for $\tau=1$, they become perfectly correlated, i.e., $\Phi = \Psi$. We are interested in the $(2 n+\nu) \times(2 n+\nu)$ random Dirac matrix
$$\mathcal{D}=\left(\begin{matrix}{}
  0 &\Phi \\
 \Psi^*  & 0
\end{matrix} \right),
$$
where $ \Psi^*$ is the complex conjugate matrix of $\Psi$. 
This matrix was first undertaken by Osborn \cite{osborn}, in the context of a study relating to quantum chromodynamics.
The spectrum of $\mathcal{D}$ consists of an eigenvalue of multiplicity $\nu$ at the origin, and $2 n$ complex eigenvalues $\left\{ \pm \zeta_k\right\}_{k=1}^n$ which come in pairs with opposite sign. 
Without loss of generality, let $(\zeta_k)_{1\le k\le n}$ be the eigenvalues with non-negative real part. The joint density function of $\zeta_1, \cdots, \zeta_{n},$ was derived by Osborn in \cite{osborn}, is proportional to   
\begin{equation}\label{density0}
\prod_{1 \leq j<k \leq n}\left|z_{j}^{2}-z_{k}^{2}\right|^{2} \cdot \prod_{j=1}^{n}\left|z_{j}\right|^{2(v+1)} \exp \left(\frac{2 \tau n \operatorname{Re}\left(z_{i}^{2}\right)}{1-\tau^{2}}\right) K_{v}\left(\frac{2 n\left|z_{j}\right|^{2}}{1-\tau^{2}}\right)
\end{equation}
for all $z_1, \cdots, z_n\in \mathbb{C}$ with ${\rm Re}(z_j)>0$ for each $j,$ where $K_v$ is the modified Bessel function of the second kind (more details on $K_v$ are given in Lemma \ref{kp} below). Here and later, the density is relative to the Lebesgue  measure on $\mathbb{C}^n.$ For details on the model, we refer to \cite{osborn} with $\mu=\sqrt{(1-\tau)} / \sqrt{(1+\tau)}, \mathcal{D}_{I I}=i \mathcal{D}, \alpha=2 /(1+\tau), A=\sqrt{1+\tau} P, B=i \sqrt{1+\tau} Q, C=i \Phi$ and $D=i \Psi^*.$
 
The random variable $\max_{1\le j\le n}|\zeta_j|$ is called the spectral radius of $\mathcal{D}.$ Unlike the Tracy-Widom distribution appears in the edge behavior in classical ensembles of Gaussian Hermitian matrices and has universality in general case, Gumbel distribution shows up as the limit of the spectral radius of non-Hermitian random matrices. 

 Pioneering work by Rider \cite{Rider03} and Rider and Sinclair \cite{Rider 14} has delved into the spectral radius of the real, complex, and quaternion Ginibre ensembles. 
 They demonstrated that appropriately scaled spectral radius converges in law to a Gumbel distribution and 
 further generalizations of this result can be found in the work of Chafai and P\'{e}ch\'{e} \cite{Chafaii}. 
Additionally, recent research by Lacroix-A-Chez-Toine \emph{ et al.} \cite{Lacroix} has explored the extreme statistics in the complex Ginibre ensemble, showing an intermediate fluctuation regime between large and typical fluctuations to the left of the mean. Their findings indicate that this interpolation also applies to general Coulomb gases. The Gumbel distribution also has a universality for the spectral radius of complex matrices $X$ with independent and identically distributed entries, where the entry $X_{1, 1}\stackrel{d}{=}n^{-1/2}\chi$ with 
$$\mathbb{E}\chi=0, \quad\mathbb{E}|\chi|^2=1, \quad \mathbb{E}\chi^2=0, \quad \mathbb{E}|\chi|^{p}\le c_p,$$ which is recently confirmed by Cipolloni {\it et al.} \cite{CipoErXu} and solves the long-standing conjecture proposed in \cite{BCC18, BCG22, Chafaii18, Chafaii}. 
 
In previous research, the limiting empirical distribution associated with the Dirac matrix $D$ for general $\tau \in [0,1)$ was established in \cite{ABK}. For the case $\tau = 1$, the empirical measures converge weakly to the Marckenko-Pastur law \cite{AnG, BSbook, MP}.

Let $(\zeta_1, \cdots, \zeta_n)$ have joint density function given in \eqref{density0}.
Given $\tau\in [0, 1),$ and $v\ge 0$ is fixed, the authors of \cite{AB} proved that  the $\max_{1\le j\le n} {\rm Re}(\zeta_j)$ converged to the Gumbel distribution and also with different scaling, $({\rm Re}(\zeta_j), {\rm Im}(\zeta_j)),$ as a point process, converged to a Poisson  process.  For the case $\tau =1,$ the largest corresponding eigenvalue converges to the Tracy-Widom distribution \cite{Johansson}.

For the particular case $\tau=0,$  Chang {\it et al.} \cite{JQ}  established a central limit theorem for $(\frac{n}{n+v})^{1/2}\max\limits_{1\le j\le n}|\zeta_j|^2.$ 
Precisely, setting $s_{n}=\frac{n(n+v)}{2n+v},$ 
  $$a(y)=(\log y)^{1/2}-(\log y)^{-1/2}\log(\sqrt{2\pi} \log y), \quad b(y)=(\log y)^{-1/2}$$ and 
  $$X_n:=\frac{1}{b(s_{n})}\left(\frac{2n \max_{1\le i\le n}|\zeta_i|^2-2\sqrt{n(n+v)}}{\sqrt{2n+v}}-a(s_{n})\right),$$ 
they proved that $X_n$ converges weakly to the Gumbel distribution $\Lambda$ as $n\to\infty,$ whose distribution function is $\Lambda(x)=e^{-e^{-x}}, x\in\mathbb{R}.$  

Motivated by \cite{JQ}, we \cite{MW24} established deviation probabilities  for $(\frac{n}{n+v})^{1/2}\max\limits_{1\le j\le n}|\zeta_j|^2,$  which offered the speed of convergence of 
$(\frac{n}{n+v})^{1/2}\max\limits_{1\le j\le n}|\zeta_j|^2$ to $1.$ 

Also, with Meng, for complex Ginibre ensembles, we investigate the exact convergence rate of its spectral radius to the Gumbel distribution \cite{L1WGinibre}, which is a natural continued work of the weak convergence. See \cite{Karoui, Rider 14, SX23} for related research on the convergence rate of extreme eigenvalues of GOE/GUE and complex Wishart to Tracy-Widom law. 

These references motivate us to offer the speed of the weak convergence of $X_n$ to the Gumbel distribution. Recall the well-known $W_1$ distance between two probabilities $\mu$ and $\nu $ on $\mathbb{R}$ (see \cite{Villani})
$$W_1(\mu, \nu)=\inf_{\pi} \int d(x, y) \pi(dx, dy),$$ where the infimum is taking 
over all couplings  $\pi$ with marginals $\mu$ and $\nu.$ For this particular case, a computable closed formula of  $W_1(\mu, \nu)$ says 
$$W_1(\mu, \nu)=\int_{-\infty}^{+\infty} |F_{\mu}(x)-F_{\nu}(x)| dx,$$
which is the starting point of this paper.

Review the joint density function of $(\zeta_1, \cdots, \zeta_{n})$ when $\tau=0$ 
\begin{equation}\label{density}
	f(z_1, \cdots, z_n)=C\prod_{1\le j<k\le n}|z_j^2-z_k^2|^2 \prod_{j=1}^n |z_j|^{2(v+1)}K_v(2n |z_j|^2).
\end{equation} 
 
Here comes first the convergence rate of $X_n$ to $\Lambda$ under $W_1$ .  
 
 \begin{thm}\label{main} 
  Let $(\zeta_1, \cdots, \zeta_{n})$ be random vector whose joint density function is given in \eqref{density} and let $X_n$ and $s_n$ be defined as above and let $\mathcal{L}(X_n)$ be the distribution of $X_n$.  
   Then, 
  $$\lim_{n\to\infty} \frac{\log s_n}{(\log\log s_n)^2}W_1\left(\mathcal{L}(X_n), \Lambda\right)=\frac{1}{2}.$$  \end{thm} 

Meanwhile, we also have the following Berry-Esseen bound. 
\begin{thm}\label{main2} 
  Let $(\zeta_1, \cdots, \zeta_{n})$ and $X_n$ and $s_n$ be defined as above. Let $F_n$ be the distribution function of $X_n.$ 
   Then,  
  $$\lim_{n\to\infty} \frac{\log s_n}{(\log\log s_n)^2}\sup_{x\in\mathbb{R}}|F_n(x)-e^{-e^{-x}}|=\frac{1}{2e}.$$ 
 \end{thm}

\begin{rmk} The form of the joint density \eqref{density} guarantees a crucial point that $X_n$ is connected to the maximum statistic of a sequence of independent random variables $(Y_k)_{k\ge 1}$ (Lemma \ref{s}) taking values on $\mathbb{R}$ and then the argument in \cite{L1WGinibre} works here. What we need are a proper upper bound and moderate deviations of Cramer type for tail probabilities  of $Y_{k}$ with $k$ large enough (Lemmas \ref{clt0} and \ref{vi}). The exact speed $\frac{(\log\log s_n)^2}{\log s_n}$ essentially comes out from the scaling $a(s_n)$ and $b(s_n)$ and we believe $(\log s_n)^{-1}$ is universal while the power of $\log \log s_n$ varies  according to the slight change of the coefficient of the term $\log(\sqrt{2\pi}\log y)$ in the definition of the function $a.$ 
\end{rmk} 
The remainder of this paper will be organized as follows : in the following section, we give some necessary lemmas for bounded $v,$ which is less difficult and then the third section is devoted to the proofs related to the case when $v$ is bounded. Once we figure out the key estimates to the proof, in  
 the forth section, we give the proof corresponding to the case $v\to\infty.$  
 The last section is focus on the proof of Theorem \ref{main2}.

For simplicity, we write $s_n$  not $s_{n, v}$ to denote $\frac{n(n+v)}{2n+v}$ even it does depend on $v$ and the same principle carries out for other notations. We use frequently $t_n=O(z_n)$ or $t_n=o(z_n)$ if $\lim_{n\to\infty}\frac{t_n}{z_n}=c\neq 0$  or $\lim_{n\to\infty}\frac{t_n}{z_n}=0,$ respectively and $t_n=\tilde{O}(z_n)$ if $\lim_{n\to\infty}\frac{t_n}{z_n}$ exists.  We also use $t_n\ll z_n$ or equivalently $z_n\gg t_n$ to represent $\lim_{n\to\infty}\frac{t_n}{z_n}=0.$ 

\section{Preliminaries} 
 In this section, we collect some key lemmas.   
 
 The first crucial one is borrowed from \cite{JQ}, originally coming from \cite{Kostlan}, which transfers the corresponding probabilities of 
 $\max_{1\le i\le n}|\zeta_i|^2$ to those related to independently random variables taking values in real line. It says 
 \begin{lem}\label{s}
	 Let $\zeta_1, \cdots, \zeta_n$ be random variables with a joint density function given by \eqref{density}.  Let $(Y_j)_{1\le j\le n}$ be independent random variables and the density function of $Y_j$ is proportional to $y^{2j+v-1} K_v(2n y) 1_{y>0}.$ Then for any symmetric function $g,$ $g(|\zeta_1^2|, \cdots, |\zeta_n^2|)$ and $g(Y_1, \cdots, Y_n)$ have the same distribution. 
\end{lem}
 
Lemma \ref{s} ensures  
$$\aligned\mathbb{P}(X_n\le x)&=\mathbb{P}\left(2n Y_{(n)}\le 2\sqrt{n(n+v)}+\sqrt{2n+v}(a(s_n)+b(s_n)x)\right)\\
&=\prod_{1\le j\le n}\mathbb{P}\left(2n Y_j\le u_n(x)\right),\endaligned $$
where for simplicity $Y_{(n)}=\max\limits_{1\le i\le n} Y_i$ and 
$$u_n(x):=2\sqrt{n(n+v)}+\sqrt{2n+v}(a(s_n)+b(s_n)x).$$ 
What we need is to investigate the tail probability of  $2n Y_j.$
Since the modified Bessel function of the second kind $K_v$ plays a crucial role in the density function of $(Y_j)_{1\le j\le n},$ we first introduce some properties of $K_v$ in the following lemma  
(see \cite{AS} and \cite{Olver}). 
\begin{lem}\label{kp} The following statements hold:
\begin{itemize} 
\item[(1).] (Formula 10.40.2 \cite{Olver} ) Suppose that $v\ge 0$ is fixed. Then, $$K_{v}(x)=\sqrt{\frac{\pi}{2 x}} e^{-x}\left[1+O\left(\frac{1}{x}\right)\right] \quad \text { as } x \rightarrow \infty.$$
\item[(2).] (Formula 10.41.4  \cite{Olver}) As $v\to\infty,$ we have that 
$$K_{v}(v x)=\sqrt{\frac{\pi}{2 v}} x^{-v}\left(1+x^{2}\right)^{-1/4} e^{-v (\sqrt{1+x^2}-\log(1+\sqrt{1+x^2})}\left[1+O\left(\frac{1}{v}\right)\right]$$ 
uniformly on $x\in (0, +\infty).$ 
\end{itemize} 	
\end{lem}

Next, we check the tail probabilities of $2n Y_j.$  
Setting \begin{equation}\label{zjconstant} Z_j:=\int_{0}^{\infty} y^{2j+v-1} K_v(y) dy,\end{equation}
we know  from the asympototic of $K_v$ with $v$ bounded that 
\begin{equation}\label{expretail}\mathbb{P}(2n Y_j>t)=\sqrt{\frac{\pi}{2}} Z_j^{-1}(1+O(t^{-1})) \int_t^{\infty} y^{2j+v-3/2} e^{-y} dy
\end{equation} 
for $t$ large enough. 
To match exactly what we need later in the proof of Theorems \ref{main} nd \ref{main2}, setting 
$j_n=o(n),$
we give a specific uniform asymptotic for $\mathbb{P}(2n Y_{j}>t)$ on $n-j_n\le j\le n.$ 

Because of the different behaviors of $K_v$ for $v$ bounded or infinite, we first provide corresponding result for $v$ bounded in this section and for the reader's convenience the lemma related to $v\to\infty$ will be postponed in the last section. We should mention that Chang {\it et al.} \cite{JQ} gave many precise calculation related to the tail probability of $Y_j.$ In our paper, for the moderate deviations of Cramer type, we adapt another way to understand the behavior of $Y_j$ for both $v$ bounded and infinite. 

\begin{lem}\label{vbexpression}
Assume that $v$ is bounded. Then, 
$$\mathbb{P}(2n Y_j>t)=\frac{1+O(t^{-1}+n^{-1})}{\Gamma(2j+v-1/2)}\int_t^{\infty} y^{2j+v-3/2} e^{-y}dy$$
\end{lem}
uniformly on $n-j_n\le j\le n$ and $t>0$ large enough.  

\begin{proof} It suffices to prove 
$$\sqrt{\frac2{\pi}} Z_{j}=\Gamma(j+v-1/2)(1+O(n^{-1})).$$
It was proved in \cite{Akemann} that 
\begin{equation}\label{zjnew} Z_j=\frac{\sqrt{\pi}\,\Gamma(2j+2v)\,\Gamma(j)}{2^{v+1}\Gamma(j+v+1/2)}\end{equation}	
for any $j\ge 1$ and $v\ge 0.$
Hence, when $j$ is sufficiently large and $v$ is bounded, the Stirling formula $$\Gamma(z)=z^{z-1/2} e^{-z} \sqrt{2\pi}(1+O(z^{-1}))$$ for $z$ large enough implies the following asymptotic 
 	$$\aligned 
 	\sqrt{\frac2{\pi}} \frac{Z_{j}}{\Gamma(2j+v-1/2)}&=\frac{\Gamma(2j+2v)\,\Gamma(j)}{2^{v+1/2}\Gamma(j+v+1/2)\Gamma(2j+v-1/2)}\\
 	&=\frac{(2j+2v)^{2j+2v-1/2} j^{j-1/2}}{2^{v+1/2}(j+v+1/2)^{j+v}(2j+v-1/2)^{2j+v-1}}(1+O(n^{-1}))\\
 	&=\frac{(1+\frac{v}{j})^{2j+2v-1/2} }{(1+\frac{v+1/2}{j})^{j+v}(1+\frac{v-1/2}{2j})^{2j+v-1}}(1+O(n^{-1}))
 	\endaligned $$
 	for all $n-j_n\le j\le n.$ 
 	Since $j=n(1+o(1))$ and $v$ is bounded, we get in further by Taylor's formula of $\log(1+t)$ for $|t|$ small enough that 
 	$$\aligned &(2j+2v-1/2)\log (1+\frac{v}{j})-(j+v)\log (1+\frac{v+1/2}{j})-(2j+v-1)\log (1+\frac{v-1/2}{2j})\\
 	&=(2j+2v-1/2)\frac{v}{j}-(j+v)\frac{v+1/2}{j}-(2j+v-1)\frac{v-1/2}{2j}+O(n^{-1})\\
 	&=O(n^{-1}).
 	\endaligned
 	$$
 	Therefore, 
 	$$ \sqrt{\frac2{\pi}} \frac{Z_{j}}{\Gamma(2j+v-1/2)}=e^{O(n^{-1})}(1+O(n^{-1}))=(1+O(n^{-1}))$$ 
 	uniformly on $n-j_n\le j\le n.$ The proof is completed. 
\end{proof}

Lemma \ref{vbexpression} gives us a hint that $2nY_j$ asymptotically obeys the Gamma distribution for the case when $v$ is bounded, which guides us to give the asymptotic for $\mathbb{P}(2n Y_j\ge u_n(x))$ via the central limit theorem of the Gamma distribution.   
\begin{lem}\label{vb} Suppose that $v$ is bounded. Given $\sqrt{n\log n}\ll j_n\ll n.$ 
 Then 
\begin{equation}\label{clt0}\mathbb{P}\left(2n Y_{n-k}>u_n(x)\right)\le\frac{1}{\widetilde{w}_n(k, x)} e^{-\frac{\widetilde{w}_n^2(k, x)}{3}}+O(n^{-1/2} \widetilde{w}_n^{-3}(k, x))
\end{equation}
uniformly on $0\le k\le j_n$ and  $x$ such that $1\ll\widetilde{w}_n(k, x).$ Here,
$$\widetilde{w}_n(k, x)=\frac{k}{\sqrt{s_n}}+a(s_n)+b(s_n) x.$$  
Furthermore, when $1\ll \widetilde{w}_n(k, x)=\tilde{O}(n^{\frac1{10}}),$ we have a precise asymptotic as
\begin{equation}\label{clt*}\mathbb{P}\left(2n Y_{n-k}>u_n(x)\right)=\frac{1+O(\widetilde{w}_n^{-2}(k, x))}{\sqrt{2\pi}\widetilde{w}_n(k, x)}e^{-\frac{\widetilde{w}_n^2(k, x)}{2}}.
\end{equation}
\end{lem}

\begin{proof}
Lemma \ref{vbexpression} enables us to write  
\begin{equation}\label{2nyj} \aligned \mathbb{P}(2n Y_j>u_n(x))&=\frac{1+O(u^{-1}_n(x)+n^{-1})}{\Gamma(2j+v-1/2)}\int_{u_n(x)}^{\infty} y^{2j+v-3/2} e^{-y} dy \\
&=(1+O(u^{-1}_n(x)+n^{-1}))\mathbb{P}({\rm Gamma}(2j+v-1/2)>u_n(x))\endaligned
\end{equation}
uniformly on $n-j_n\le j\le n$ and  $x\ge -\ell_n$ with $\ell_n$ satisfying $1\ll u_n(-\ell_n).$  Here, we use the notation ${\rm Gamma}(2j+v-1/2)$ to represent a Gamma distribution with parameter $2j+v-1/2.$
For any $v\ge 0$ bounded, set $m=[v-1/2].$  
By the additivity of independent Gamma distributions, it follows 
$$\aligned\mathbb{P}({\rm Gamma}(2j+m)>u_n(x))&\le \mathbb{P}({\rm Gamma}(2j+v-1/2)>u_n(x))\\
&\le \mathbb{P}({\rm Gamma}(2j+m+1)>u_n(x)).
\endaligned$$

On the one hand, the non-uniform Berry-Esseen bounds ((8.1) in \cite{Stein}) guarantees 
$$\aligned &\quad \mathbb{P}\left(\frac{{\rm Gamma}(2j+\kappa)-(2j+k)}{\sqrt{2j+\kappa}}>t\right)=1-\Phi(t)+O(n^{-1/2} t^{-3})
\endaligned$$ for all $n-j_n\le j\le n$ and $\kappa$ bounded integer and $t$ large enough while $n$ is sufficiently large. 
Thereby, taking $\kappa=m+1$ and $\kappa=m,$ respectively, one gets  
$$\aligned &\quad \mathbb{P}\left(\frac{{\rm Gamma}(2j+v-1/2)-(2j+v-1/2)}{\sqrt{2j+v-1/2}}>t\right)\\
&\le \mathbb{P}\left(\frac{{\rm Gamma}(2j+m+1)-(2j+m+1)}{\sqrt{2j+m+1}}>t(1+O(n^{-1}))+O(n^{-1/2})\right)\\
&=1-\Phi\left(t+O(t n^{-1}+n^{-1/2})\right)+O(n^{-1/2} t^{-3})
\endaligned$$
and 
$$\aligned &\quad \mathbb{P}\left(\frac{{\rm Gamma}(2j+v-1/2)-(2j+v-1/2)}{\sqrt{2j+v-1/2}}>t\right)\\
&\ge \mathbb{P}\left(\frac{{\rm Gamma}(2j+m)-(2j+m)}{\sqrt{2j+m}}>t(1+O(n^{-1}))+O(n^{-1/2})\right)\\
&=1-\Phi\left(t+O(t n^{-1}+n^{-1/2})\right)+O(n^{-1/2} t^{-3}).
\endaligned$$ 
Consequently, it holds 
\begin{equation}\label{asymll} \aligned & \quad \mathbb{P}\left(\frac{{\rm Gamma}(2j+v-1/2)-(2j+v-1/2)}{\sqrt{2j+v-1/2}}>t\right)\\
&=1-\Phi\left(t+O(t n^{-1}+n^{-1/2})\right)+O(n^{-1/2} t^{-3})\endaligned \end{equation}
Having \eqref{2nyj} and \eqref{asymll} at hands and setting 
$$w_n(k, x):=\frac{u_n(x)-(2(n-k)+v-1/2)}{\sqrt{2(n-k)+v-1/2}}$$ for $0\le k\le j_n,$ 
 we can write 
$$\aligned &\quad \mathbb{P}(2n Y_{n-k}>u_n(x))\\
&=(1+O(n^{-1}+u_n^{-1}(k, x)))\mathbb{P}\left(\frac{{\rm Gamma}(2(n-k)+v-1/2)-(2(n-k)+v-1/2)}{\sqrt{2(n-k)+v-1/2}}>w_n(k, x)\right)\\
&=(1+O(n^{-1}+u_n^{-1}(k, x)))\left(1-\Phi\left(w_n(k, x)+O(w_n(k, x) n^{-1}+n^{-1/2})\right)\right)+O(n^{-1/2} w_n^{-3}(k, x)).
\endaligned$$
Simple calculus and the condition $\sqrt{n\log n}\ll j_n\ll n$ entails  
$$\aligned w_n(k, x)&=\frac{2\sqrt{n(n+v)}-(2n-2k+v-1/2)+\sqrt{2n+v}\,(a(s_n)+b(s_n) x) \,}{\sqrt{2n-2k+v-1/2}}\\
&=\left(\frac{k}{\sqrt{s_n}}+a(s_n)+b(s_n) x\right)(1+O(k\, n^{-1}))+O(n^{-1/2})\\
&=\left(\frac{k}{\sqrt{s_n}}+a(s_n)+b(s_n) x\right)(1+O(k\, n^{-1})).
\endaligned$$
Thus, with the notation of $\widetilde{w}_n(k, x),$
 one gets 
$$w_n(k, x)+O(w_n(k, x) n^{-1}+n^{-1/2})=\widetilde{w}_n(k, x)(1+O(k n^{-1}+n^{-1/2}\widetilde{w}_n(k, x)^{-1})).$$
Recall Mills ratio \begin{equation}\label{normal}1-\Phi(t)=\frac{1}{\sqrt{2\pi} \,t} e^{-\frac{t^2}{2}}(1+O(t^{-2}))\end{equation}
for $t>0$ large enough. Here, $\Phi$ is the distribution function of a standard normal. 

Eventually, uniformly on $n-j_n\le j\le n$ and $x\ge -\ell_n$, we have 
$$\aligned \mathbb{P}\left(2n Y_j>u_n(x)\right)&=\frac{1+O(k n^{-1}+\widetilde{w}_n^{-2}(k, x))}{\sqrt{2\pi}\widetilde{w}_n(k, x)} e^{-\frac{\widetilde{w}_n^2(k, x)}{2}\left(1+O( k n^{-1})\right)}+O(n^{-1/2} \widetilde{w}_n(k, x)^{-3})\\
&\le\widetilde{w}_n^{-1}(k, x)e^{-\frac{1}{3}\widetilde{w}_n^2(k, x)}+O(n^{-1/2} \widetilde{w}_n^{-3}(k, x))\endaligned
$$  when $n$ is sufficiently large. 

As far as \eqref{clt*} is concerned, we borrow Theorem 2 on page 217 of the book by Petrov \cite{Petrov}, which entails that 
$$\aligned \quad \mathbb{P}\left(\frac{{\rm Gamma}(2j+\kappa)-(2j+k)}{\sqrt{2j+\kappa}}>t\right)=(1+O((t^3+1) n^{-1/2}))(1-\Phi(t)) 
\endaligned$$
uniformly on $n-j_n\le j\le n$ and $0<t<\tau_n=o(n^{1/6}).$ 
The same argument for \eqref{asymll} implies 
\begin{equation}\label{something}\aligned &\quad \mathbb{P}(2n Y_{n-k}>u_n(x))\\
&=(1+O(n^{-1}))\mathbb{P}\left(\frac{{\rm Gamma}(2(n-k)+v-1/2)-(2(n-k)+v-1/2)}{\sqrt{2(n-k)+v-1/2}}>w_n(k, x)\right)\\
&=(1+O(\widetilde{w}_n^3(k, x)n^{-1/2}))\left(1-\Phi\left(w_n(k, x)+O(w_n(k, x) n^{-1}+n^{-1/2})\right)\right)\\
&=\frac{1}{\sqrt{2\pi}\widetilde{w}_n(k, x)}e^{-\frac{\widetilde{w}_n^2(k, x)}{2}\left(1+O( k n^{-1})\right)}\big(1+O(\widetilde{w}_n^3(k, x)n^{-1/2}+\widetilde{w}_n^{-2}(k, x)+k n^{-1})\big)\\
&=\frac{1}{\sqrt{2\pi}\widetilde{w}_n(k, x)}e^{-\frac{\widetilde{w}_n^2(k, x)}{2}}\big(1+O(\widetilde{w}_n^3(k, x)n^{-1/2}+\widetilde{w}_n^{-2}(k, x)+\widetilde{w}_n^2(k, x)k n^{-1})\big)\\
&=\frac{1+O(\widetilde{w}_n^{-2}(k, x))}{\sqrt{2\pi}\widetilde{w}_n(k, x)}e^{-\frac{\widetilde{w}_n^2(k, x)}{2}}.
\endaligned
\end{equation} 
Here, the last but second equality is true because that the condition $\widetilde{w}_{n}(k, x)=\tilde{O}(n^{\frac{1}{10}})$ implies $k=\tilde{O}(n^{3/5})$ and then $\widetilde{w}_{n}^2(k, x) \frac k n=o(1)$ 
and the last equality follows from 
$$O(\widetilde{w}_n^3(k, x)n^{-1/2}+\widetilde{w}_n^{-2}(k, x)+\widetilde{w}_n^2(k, x)k n^{-1})=O(\widetilde{w}_n^{-2}(k, x)).$$
The proof is then completed. 
\end{proof}

Reviewing the proof in \cite{L1WGinibre}, we give another simpler upper bound for $\mathbb{P}(2n Y_{n-k}\ge u_n(x))$ with $\widetilde{w}_n(k, x)\ge n^{1/10}(1+o(1)),$ which is rougher but enough for some estimates. Now $\widetilde{w}_n(k, x)\ge n^{1/10}$ ensures 
$\widetilde{w}_n(k, x)\gg \log \widetilde{w}_n(k, x)+\log n,$ which implies  
$$\widetilde{w}_n^2(k, x)\sqrt{n}\;e^{-\frac{\widetilde{w}_n^2(k, x)}{3}}\ll 1.$$ 
Therefore, \eqref{clt0} in Lemma \ref{vb} indicates the following result.

\begin{lem}\label{rough} For $x$ and $k$ such that $\widetilde{w}_n(k, x)\ge n^{\frac1{10}}(1+o(1)),$ we have 
$$\mathbb{P}(2n Y_{n-k}\ge u_n(x))=O\left(n^{-1/2} \widetilde{w}_{n}^{-3}(k, x)\right)$$ for sufficiently large $n.$	
\end{lem}

Next, we summarize some properties of $\widetilde{w}_n(k, x)$ in the following lemma, whose proof is exactly the same as that of Lemma 2.4 \cite{L1WGinibre} and is omitted here. 
\begin{lem}\label{wnkx} Assume the same condition as Lemma \ref{vb}. Given any $z_n$ such that $\widetilde{w}_n(k, z_n)$ is large enough. 
\begin{itemize}
\item[(1).] For any $m>1,$ it follows 
$$\int_{z_n}^{\infty}\widetilde{w}^{-m}_n(k, x)dx=\frac{\sqrt{\log s_n}}{m-1} \widetilde{w}^{-m+1}_n(k, z_n).$$ 	
\item[(2).] For any $c, m>0,$ we have 
$$\int_{z_n}^{\infty} \widetilde{w}^{-m}_n(k, x)e^{-c \,\widetilde{w}^2_n(k, x)}dx=\frac{\sqrt{\log s_n}}{2c\,\widetilde{w}_n^{m+1}(k, z_n)}\, e^{-c \,\widetilde{w}^2_n(k, z_n)}\left(1+O(\widetilde{w}_n^{-2}(k, z_n))\right)$$
\item[(3).] Given $c, m>0$ and $L\ge 0.$   Suppose $1\ll \widetilde{w}_n(L, x_n)\ll\sqrt{n}.$ We have 
$$\sum_{k=L}^{+\infty} \widetilde{w}_n^{-m}(k, z_n)e^{-c \widetilde{w}_n^2(k, z_n)}=\frac{\sqrt{s_n} e^{-c\widetilde{w}_n^2(L, z_n)}}{2 c\, \widetilde{w}_n^{(m+1)}(L, z_n)}(1+O(\widetilde{w}_n^{-2}(L, z_n)+\widetilde{w}_n(L, z_n)n^{-1/2})). $$ 
\item[(4).] Given $m>1$ and $L\ge 0$ such that $1\ll \widetilde{w}_n(L, x_n)$. Then 
$$\sum_{k=L}^{+\infty} \widetilde{w}_n^{-m}(k, z_n)=\frac{\sqrt{s_n}}{m-1}\widetilde{w}_n^{-m+1}(L, z_n)(1+O(\widetilde{w}_n^{-1}(L, z_n)n^{-1/2})).$$
\end{itemize}
\end{lem}

\section{Proof of Theorem \ref{main} when $v$ is bounded}  
This section is devoted to the proof of Theorem \ref{main} for the case when $v$ is bounded. 
 
Remember  
$$\mathbb{P}(X_n\le x)=\mathbb{P}(2n Y_{(n)}\le u_{n}(x)),$$
where $$u_{n}(x)=2\sqrt{n(n+v)}+\sqrt{2n+v}(a(s_{n})+b(s_{n}) x)$$ 
and $s_n=\frac{n(n+v)}{2n+v}.$
Recall 
$$\widetilde{w}_{n}(k, x)=a(s_{n})+b(s_{n}) x+\frac{k}{\sqrt{s_n}}$$
and $a(y)=\sqrt{\log y}-\frac{\log(\sqrt{2\pi}\log y)}{\sqrt{\log y}}$ and $b(y)=(\log y)^{-1/2}.$   
By the closed formula of $W_1$ distance, we see 
$$W_1\left(\mathcal{L}(X_n), \Lambda\right)=\int_{-\infty}^{+\infty} |\mathbb{P}(2 nY_{(n)}\le x)-e^{-e^{-x}}|dx.$$ 
As the author did with Meng in \cite{L1WGinibre}, we still choose proper $\ell_1(n)$ and $\ell_2(n)$ with $1\ll l_1(n), l_2(n)\ll \log s_n $ to separate the integral into three parts as follows 
\begin{equation}\label{total}\aligned &\quad \int_{-\infty}^{+\infty} |\mathbb{P}(2 nY_{(n)}\le u_n(x))-e^{-e^{-x}}|dx\\
&=\int_{-\infty}^{-\ell_{1}(n)}|\mathbb{P}(2 nY_{(n)}\le u_n(x))-e^{-e^{-x}}|dx+\int^{\ell_2(n)}_{-\ell_{1}(n)}|\mathbb{P}(2 nY_{(n)}\le u_n(x))-e^{-e^{-x}}|dx\\
&\quad +\int_{\ell_{2}(n)}^{\infty}|\mathbb{P}(2 nY_{(n)}\le u_n(x))-e^{-e^{-x}}|dx\\ 
&=:{\rm I}+{\rm I\!I}+{\rm I\!I\!I}.
\endaligned 	
\end{equation}
 
The definitions of functions $a$ and $b$ as well as $s_{n}$  tell us that 
$$a(s_{n})=O(\sqrt{\log s_n}), \quad b(s_{n})=O((\log s_n)^{-1/2}),$$ which implies the infinity of $\widetilde{w}_n(k, x)$ for all $0\le k\le j_n$ and $x\ge-\ell_{1}(n).$
Thus, 
Lemma \ref{vb} is always available in this section.   

To obtain $$\int_{-\infty}^{+\infty} |\mathbb{P}(2 nY_{(n)}\le u_n(x))-e^{-e^{-x}}|dx=\frac{(\log\log s_n)^2}{2\log s_n}(1+o(1)),$$ 
it suffices to verify the following three facts:
\begin{enumerate}
\item ${\rm I}\ll \frac{(\log\log s_n)^2}{\log s_n};$
\item ${\rm I\!I}=\frac{(\log\log s_n)^2}{2\log s_n}(1+o(1));$
\item ${\rm I\!I\!I}\ll \frac{(\log\log s_n)^2}{\log s_n}.$	
\end{enumerate}
 
\subsection{Estimate on {\rm I}} 
Given two integers $m_1(n), m_2(n)\ge 2.$ Let $-x_0$ be the unique solution to $u_n(x)=0,$ 
which means $x_0=\log s_n+2\sqrt{s_n\log s_n}.$ The choice of $\ell_1(n)$  makes sure both $\mathbb{P}(2n Y_{(n)}\le u_n(x))$ and $e^{-e^{-x}}$ are around zero, which allows us to use the inequality $|t-s|\le |t|+|s|$ without any hesitation. 
Now $$\mathbb{P}(2n Y_j\le u_n(x))=0, \quad x\le x_0,$$ we see with the notation $\alpha_n=\frac{a(s_n)}{b(s_n)}$ that 
$$\aligned {\rm I}&=\int_{-\infty}^{-\ell_{1}(n)} |\mathbb{P}(2n Y_{(n)}\le u_n(x))-e^{-e^{-x}}|dx\\
&\le \int_{-\infty}^{-\ell_{1}(n)} e^{-e^{-x}}dx+\int_{-x_0}^{-\alpha_n} \mathbb{P}(2n Y_{(n)}\le u_n(x)) dx+\int_{-\alpha_n}^{-\ell_{1}(n)} \mathbb{P}(2n Y_{(n)}\le u_n(x)) dx\\
&\le \int_{e^{\ell_{1}(n)}}^{\infty} t^{-1} e^{-t} dt+x_0 \prod_{k=0}^{m_1(n)}\mathbb{P}\big(2n Y_{n-k}\le u_n\big(-\alpha_n\big)\big)+\alpha_n\prod_{k=0}^{m_2(n)}\mathbb{P}(2n Y_{n-k}\le u_n(-\ell_1(n))) \\
&\le e^{-\ell_1(n)-e^{\ell_1(n)}}+x_0\mathbb{P}^{m_1(n)}\big(2n Y_{n-m_1(n)}\le u_n\big(-\alpha_n\big)\big)\\
&\quad \quad \quad \quad\quad\quad \quad+\log s_n\, \exp\big\{-\sum_{k=0}^{m_2(n)}\mathbb{P}(2n Y_{n-k}\ge u_n(-\ell_1(n)))\big\}, \endaligned  $$
 where for the last inequality we use the facts $\alpha_n<\log s_n$ and $\log (1-t)<-t$ and the monotonicity of $\mathbb{P}(2n Y_k\le  u_n(x))$ on $x$ and on $k,$ which was verified in \cite{JQ}.
Choose particularly \begin{equation}\label{elln}\ell_1(n)=\log(2\log\log s_n),\end{equation} and $m_2(n)$ such that 
$$\widetilde{w}_n(m_2(n), -\ell_1(n))=2\sqrt{\log s_n}+O(s_n^{-1/2}).$$ It follows $$\aligned \widetilde{w}_{n}(0, -\ell_1(n))&=\sqrt{\log s_n}\big(1+o(1)\big);\\
\widetilde{w}_{n}^2(0, -\ell_1(n))&=\log s_n-2(\ell_1(n)+\log(\sqrt{2\pi}\log s_n))+o(1).\endaligned $$ 
This, together with the help of Lemmas \ref{vb} and \ref{wnkx}, implies that  
  $$\aligned &\quad \sum_{k=0}^{m_2(n)}\mathbb{P}(2n Y_{n-k}\ge u_n(-\ell_1(n)))\\
  &=\sum_{k=0}^{m_2(n)}\frac{1+o(1)}{\sqrt{2\pi} \widetilde{w}_{n}(k, -\ell_1(n))}e^{-\frac{\widetilde{w}_{n}^2(k, -\ell_1(n))}{2}}\\
  &=\frac{(1+o(1))\sqrt{s_n}}{\sqrt{2\pi}\widetilde{w}_{n}^2(0, -\ell_1(n))}e^{-\frac{\widetilde{w}_{n}^2(0, -\ell_1(n))}{2}}-\frac{(1+o(1))\sqrt{s_n}}{\sqrt{2\pi}\widetilde{w}_{n}^2(m_2(n)+1, -\ell_1(n))}e^{-\frac{\widetilde{w}_{n}^2(m_2(n)+1, -\ell_1(n))}{2}}\\
  &=2\log\log s_n+O(n^{-3/2}(\log n)^{-1}).\endaligned$$  
This leads to 
 $$\aligned 
\log s_n\, \exp\left\{-\sum_{k=0}^{m_2(n)}\mathbb{P}(2n Y_{n-k}\ge u_n(-\ell_1(n)))\right\}&=\exp\{-\log\log s_n+O(n^{-3/2}(\log n)^{-1})\}\\
&=\frac{1}{\log s_n}\left(1+o\left(1\right)\right). 
\endaligned $$
For the probability 
$\mathbb{P}\big(2n Y_{n-m_1(n)}\le u_n\big(-\alpha_n\big)\big),$ the second expression in \eqref{something} shows 
$$\aligned &\quad \mathbb{P}\big(2n Y_{n-m_1(n)}\ge u_n\big(-\alpha_n\big)\big)\\
&=(1+O(n^{-1/2} \widetilde{w}_n^3(m_1(n),-\alpha_n)+m_1(n) n^{-1}))\\
&\quad \times \left(1-\Phi\left(w_n(m_1(n), -\alpha_n)+O(w_n(m_1(n), -\alpha_n) n^{-1}+n^{-1/2})\right)\right).
\endaligned $$
Choose $m_1(n)=[\sqrt{s_n}]$ to make $\widetilde{w}_{n}(m_1(n), -\alpha_n)=1+O(s_n^{-1/2}).$ As a consequence, it holds 
$$\aligned \mathbb{P}\big(2n Y_{n-m_1(n)}\ge u_n\big(-\alpha_n\big)\big)&=(1+O(n^{-1/2}))\left(1-\Phi\left(1+n^{-1/2})\right)\right)
\ge \frac13,
\endaligned $$
which implies 
$$\mathbb{P}^{m_1(n)}\big(2n Y_{n-m_1(n)}\le u_n\big(-\alpha_n\big)\big)\le (\frac23)^{m_1(n)}.$$
Combining all these estimates together, we see 
\begin{equation}\label{I} {\rm I}=\frac{1}{\log s_n}\left(1+o\left(1\right)\right)\ll \frac{(\log\log s_n)^2}{\log s_n}.
\end{equation} 

\subsection{The asymptotic on {\rm I\!I}} 
Now, we investigate the order of the second term {\rm I\!I}, which is 
$${\rm I\!I}=\int_{-\ell_1(1)}^{\ell_2(n)} |\mathbb{P}(2n Y_{(n)}\le u_n(x))-e^{-e^{-x}}| dx$$ and then 
$${\rm I\!I}=\int_{-\ell_1(1)}^{\ell_2(n)} |e^{\mu_n(x)}-e^{-e^{-x}}| dx,$$
where $$\mu_n(x)=\sum_{k=1}^n\log \mathbb{P}(2n Y_k\le u_n(x)).$$ 
 If we are able to prove  
\begin{equation}\label{munx}
\mu_n(x)+e^{-x}=e^{-x}q_n(x)=o(1) \quad \text{uniformly on} \, x\in (-\ell_1(n), \ell_2(n)),
\end{equation}
we can have via Taylor's formula that  
\begin{equation}\label{iilower}\aligned 
{\rm I\!I}=\int_{-\ell_1(n)}^{\ell_2(n)} e^{-e^{-x}}| e^{\mu_n(x)+e^{-x}}-1|\,dx=(1+o(1))\int_{-\ell_1(n)}^{\ell_2(n)}e^{-e^{-x}-x}|q_n(x)|dx.
\endaligned 	
\end{equation}
Now we are going to explore what $q_n(x)$ is. Indeed, since $\log \mathbb{P}(2n Y_k\le u_n(x))$ is decreasing on $k$ and is non-positive,  it holds naturally with $j_n=[n^{3/5}]$  that
$$ \aligned \mu_n(x)&\ge \sum_{k=0}^{j_n-1} \log \mathbb{P}(2n Y_{n-k}\le u_n(x))+(n-j_n)\log \mathbb{P}(2n Y_{n-j_n}\le u_n(x));\\
\mu_n(x)&\le \sum_{k=0}^{j_n-1} \log \mathbb{P}(2n Y_{n-k}\le u_n(x)).
\endaligned$$ 

We first claim that $$\mathbb{P}(2nY_k\ge u_n(x))=o(s_n^{-1/3})$$ for any $1\le k\le n,$ which is correct once 
 $$\mathbb{P}(2nY_n\ge u_n(x))=o(s_n^{-1/3}).$$ Indeed, Lemma \ref{clt0} says 
 $$\aligned \mathbb{P}(2nY_n\ge u_n(x))&=\frac{1+o(1)}{\sqrt{2\pi}\widetilde{w}_n(0, x)} e^{-\frac{\widetilde{w}_n^2(0, x)}{2}}.
 \endaligned $$
 We first check the precise form of $\widetilde{w}_n(0, x).$ By definition, 
$$\widetilde{w}_n(0, x)=\sqrt{\log s_n}+\frac{x-\ell_2(n)}{\sqrt{\log s_n}}=\sqrt{\log s_n}-t_n(x),$$ 
where for simplicity $$t_n(x):=\frac{\ell_2(n)-x}{\sqrt{\log s_n}}$$ and then  
\begin{equation}\label{expwn0x} \aligned e^{-\frac12\widetilde{w}_n(0, x)^2}&=\frac{\sqrt{2\pi} \log s_n}{\sqrt{s_n}} e^{-x} e^{-\frac{t_n^2(x)}{2}}.
\endaligned \end{equation}
Thereby, it follows 
 $$\aligned \mathbb{P}(2nY_n\ge u_n(x))&=O\big(\frac{\sqrt{\log s_n}}{\sqrt{s_n}}e^{-x}\big)=o(s_n^{-1/3})
 \endaligned $$
 uniformly on $x\in (-\ell_1, \ell_2(n)).$
 Thus, use the Taylor expansion to obtain
$$ \aligned \log \mathbb{P}(2n Y_{n-k}\le u_n(x))&=\log(1- \mathbb{P}(2n Y_{n-k}\ge u_n(x))) \\
&=-\mathbb{P}(2n Y_{n-k}\ge u_n(x))(1+o(s_n^{-1/3})),\endaligned$$ 
which ensures that $\mu_n(x)$ satisfies 
$$\aligned \mu_n(x)&\ge (-\sum_{k=0}^{j_n-1} \mathbb{P}(2n Y_{n-k}> u_n(x))-(n-j_n)\mathbb{P}(2n Y_{n-j_n}> u_n(x)))(1+o(s_n^{-1/3}));\\
\mu_n(x)&\le-\sum_{k=0}^{j_n-1} \mathbb{P}(2n Y_{n-k}> u_n(x)).\endaligned$$

Now $j_n=[n^{3/5}],$ so that $\widetilde{w}_n(j_n, x)=n^{1/10}+O(\sqrt{\log s_n})$ and $\widetilde{w}_n(k, x)=\tilde{O}(n^{\frac1{10}})$ uniformly on $x\in (-\ell_1(n), \ell_2(n))$ and $0\le k\le j_n.$ 
Lemma \ref{clt0} tells 
\begin{equation}\label{un1}\aligned n\mathbb{P}(2n Y_{n-j_n}>u_{n}(x))
&=\frac{n(1+o(1))}{\sqrt{2\pi}\widetilde{w}_{n}(j_n, x)} e^{-\frac{\widetilde{w}_{n}^{2}(j_n, x)}{2}}=o(n^{\frac{9}{10}} e^{-\frac13 n^{\frac{1}{5}}}).
\endaligned \end{equation}

Setting $$\aligned \lambda_n(x)&= \sum_{k=0}^{j_n-1} \mathbb{P}(2n Y_{n-k}\ge u_n(x)), \endaligned $$
 Lemma \ref{clt0} continues to ensure  
\begin{equation}\label{lambdax1}\aligned \lambda_n(x)=\sum_{k=0}^{j_n-1}\frac{1+O(\widetilde{w}_n^{-2}(k, x))}{\sqrt{2\pi} \widetilde{w}_n(k, x)}e^{-\frac{1}{2} \widetilde{w}_n^2(k, x)},\endaligned\end{equation} 
which with Lemma \ref{wnkx} and \eqref{expwn0x} helps us to get 
$$\aligned \lambda_n(x)&=\frac{\sqrt{s_n}(1+O((\log s_n)^{-1}))}{\sqrt{2\pi}\widetilde{w}_n^2(0, x)} e^{-\frac{1}{2} \widetilde{w}_n^2(0, x)}-\frac{\sqrt{s_n}(1+o(1))}{\sqrt{2\pi}\widetilde{w}_n^2(j_n, x)} e^{-\frac{1}{2} \widetilde{w}_n^2(j_n, x)}\\
&=e^{-\frac{t_n^2(x)}{2}-x}(1-\frac{t_n(x)}{\sqrt{\log s_n}})^{-2}(1+O((\log s_n)^{-1})).\endaligned 
$$
The fact $t_n(x)=O\left(\frac{\ell_2(n)}{\sqrt{\log s_n}}\right)=o(1)$ implies $\frac{t_n(x)}{\sqrt{\log s_n}}\ll t_n^2(x),$ which allows us to apply the Taylor formula on $e^{-\frac{t_n^2(x)}{2}}$ and to write $$(1-\frac{t_n(x)}{\sqrt{\log s_n}})^{-2}=1+O(\frac{\ell_2(n)}{\log s_n})$$ and then we have 
\begin{equation}\label{un2}\aligned \lambda_n(x)&=e^{-x}(1-\frac{t_n^2(x)}{2}+O(t_n^{4}(x)))(1+O(\frac{\ell_2(n)}{\log s_n}))\\
&=e^{-x}(1-\frac{t_n^2(x)}{2})(1+O(\frac{\ell_2(n)}{\log s_n})).\endaligned \end{equation} 
The fact $$e^{-x}\ge e^{-\ell_2(n)}=O((\log s_n)^{-1})\gg n^{9/10}e^{-\frac{1}{3} n^{1/5}} \log s_n$$
and the two expressions \eqref{un1} and \eqref{un2} imply 
$$\mu_x=-\lambda_n(x)(1+o(s_n^{-1/3}))=-e^{-x}(1-\frac{t_n^2(x)}{2})(1+O(\frac{\ell_2(n)}{\log s_n}))$$
and moreover 
$$\aligned \mu_n(x)+e^{-x}=-\frac{t_n^2(x)}{2}e^{-x}(1+o(1)). \endaligned $$
This roughly gives 
$$\aligned\mu_n(x)+e^{-x}&=e^{-x} O\big(t_n^2(x)\big)
\le e^{\ell_1(n)}O\left(\frac{(\log \log s_n)^2}{\log s_n}\right)=O\left(\frac{(\log \log s_n)^3}{\log s_n}\right)=o(1)\endaligned $$   
uniformly on $x\in (-\ell_1(n), \ell_2(n)),$ 
which in fact is the hidden reason why we chose $\ell_1(n)=\log(2\log \log s_n)$ and $\ell_2(n)=\log(\sqrt{2\pi}\log s_n).$ 
Applying \eqref{iilower} with $q_n(x)=-\frac{t_n^2(x)}{2}(1+o(1)),$ we finally have 
$$\aligned{\rm I\!I}&=\frac{1+o(1)}{2\log s_n}\int_{-\ell_1(n)}^{\ell_2(n)} e^{-e^{-x}} e^{-x}|\ell_2(n)-x|^2 dx.  
\endaligned $$ 
By the substitution $t=e^{-x}$, it follows 
\begin{equation}\label{inteii}\aligned &\quad \int_{-\infty}^{+\infty} e^{-e^{-x}} e^{-x} (\ell_2(n)-x)^2 dx\\&=\int_{0}^{\infty} e^{-t}(\ell_2(n)+\log t)^2 dt \\
&=\ell_2^2(n)\int_{0}^{\infty} e^{-t}dt+2\ell_2(n)\int_{0}^{\infty} e^{-t}\log t dt+\int_{0}^{\infty} e^{-t}\log^2 t dt\\
&=\ell_2^2(n)(1+o(1)).
\endaligned \end{equation} 
Here, the last equality holds because the last two integrals in the third expression of \eqref{inteii} converge and $\ell_2(n)$ is large enough. 

Thus, it follows 
$$\int_{-\ell_1(n)}^{\ell_2(n)} e^{-e^{-x}} e^{-x}|\ell_2(n)-x|^2 dx=\ell_2^2(n)(1+o(1))$$
since both $\ell_1(n)$ and $\ell_2(n)$ tend to the infinity and then 
$$\aligned{\rm I\!I}&=\frac{(1+o(1))\ell_2^2(n)}{2\log s_n}=\frac{(\log \log s_n)^2}{2\log s_n}(1+o(1)).  
\endaligned $$

\subsection{Estimate on {\rm I\!I\!I}} 

Recall
$${\rm I\!I\!I}=\int_{\ell_2(n)}^{\infty}|\mathbb{P}(2n Y_{(n)}\le u_n(x))-e^{-e^{-x}}|dx=\int_{\ell_2(n)}^{\infty}|e^{\mu_n(x)}-e^{-e^{-x}}|dx.
$$
The elementary inequality $1-e^{-t}\le t$ entails   
$$
	{\rm I\!I\!I}\le\int_{\ell_2(n)}^{\infty}(-\mu_n(x))dx+\int_{\ell_2(n)}^{\infty}e^{-x} dx=e^{-\ell_2(n)}-\int_{\ell_2(n)}^{\infty}\mu_n(x)dx.
$$
Thus, it remains to prove that 
$$	\int_{\ell_2(n)}^{\infty}(-\mu_n(x))dx\ll \frac{(\log\log s_n)^2}{\log s_n},
$$
which is true if we are able to prove 
\begin{equation}\label{III01} 
	\int_{\ell_2(n)}^{\infty}\big(\sum_{k=0}^{j_n-1}\mathbb{P}(2nY_{n-k}\ge u_n(x))+(n-j_n) \mathbb{P}(2nY_{n-j_n}\ge u_n(x))\big)dx\ll \frac{(\log\log s_n)^2}{\log s_n}
\end{equation}
for some $\sqrt{n\log s_n}\ll j_n\ll n.$ 
We choose $j_n=[n^{3/4}\log n],$ which makes $\widetilde{w}_n(j_n, x)\ge n^{\frac{1}{10}}(1+o(1)).$
Thus Lemma \ref{rough} entails 
$$\aligned \mathbb{P}(2n Y_{n-j_n}>u_{n}(x))
&=O(n^{-1/2} \widetilde{w}_{n}^{-3}(j_n, x))\\
\endaligned  $$ 
and thanks to Lemma \ref{wnkx}, we have 
\begin{equation}\label{summe0}\aligned n \int_{\ell_2(n)}^{\infty}\mathbb{P}(2n Y_{n-j_n}>u_{n}(x))=O\left(\frac{\sqrt{n\log s_n} }{\widetilde{w}_{n}^2(j_n, \ell_2(n))}\right)=O(\frac{n^{3/2}\sqrt{\log n}}{j_n^2})=O((\log n)^{-3/2}).
\endaligned  \end{equation}

Meanwhile, with the help of Lemmas \ref{vb}, \ref{rough} and \ref{wnkx}, and choosing $p_n=n^{\frac1{10}}\sqrt{\log s_n},$ we have  
\begin{equation}\label{summe}\aligned &\quad \sum_{k=0}^{j_n-1}\int_{\ell_2(n)}^{\infty}\mathbb{P}(2n Y_{n-k}>u_{n}(x)) dx\\
&=\sum_{k=0}^{j_n-1}\int_{\ell_2(n)}^{p_n}\mathbb{P}(2n Y_{n-k}>u_{n}(x)) dx+\sum_{k=0}^{j_n-1}\int_{p_n}^{+\infty}\mathbb{P}(2n Y_{n-k}>u_{n}(x)) dx\\
&\le\sum_{k=0}^{j_n-1}\int_{\ell_2(n)}^{p_n} \widetilde{w}_{n}^{-1}(k, x) e^{-\frac{\widetilde{w}_{n}^2(k, x)}{2}}dx+O\left(\sum_{k=0}^{j_n-1}n^{-1/2} \int_{p_n}^{\infty}\widetilde{w}_{n}^{-3}(k, x) dx \right)\\
&=\frac{\sqrt{s_n\log s_n}}{\widetilde{w}_{n}^3(0, \ell_2(n))}e^{-\frac{\widetilde{w}_{n}^2(0, \ell_2(n))}{2}}+O\left(\frac{\sqrt{s_n\log s_n}}{\sqrt{n}\widetilde{w}_{n}(0, p_n)}\right)\\
&=O((\log s_n)^{-1}).
\endaligned  \end{equation}

Here, the validity of the last equality owes a big thank to the fact 
 $$\widetilde{w}_{n}(0, \ell_2(n))=\sqrt{\log s_n}.$$ 

The expressions \eqref{summe0} and \eqref{summe} confirm the estimate \eqref{III01} we want.

\section{Proof of Theorem \ref{main} when $v$ tends to the infinity.}
In this section, we are going to verify Theorem \ref{main} for the case when $v$ tends to the infinity. 

When we outline the proof of Theorem \ref{main} for $v$ bounded, we realize that the only thing need to do for the case when $v$  is infinite is to provide a similar lemma as \ref{vb}. 
As showed below, we obtain first 
$$\mathbb{P}(2n Y_{n-k}\ge u_n(x))=\frac{1+\tilde{O}(n^{-1})}{\sqrt{2\pi}}\int_{\widetilde{w}_n(k, x)(1+O(k/n))}^{\infty} e^{-v(\tau_j(\mu_j+\frac{\sqrt{2j+v}}{v} z))-\tau_j(\mu_j))} dz,
$$
where $k=n-j.$
Hence, according to the argument and all estimates made in the precedent section, we only need a precise expression as \eqref{clt*} for $\mathbb{P}(2n Y_{n-k}>u_n(x))$ when $1\ll \widetilde{w}_n(k, x)=\tilde{O}(n^{\frac{1}{10}})$ and meanwhile a proper, not necessarily sharp, upper bound  for $\mathbb{P}(2n Y_{n}>u_n(x))$ is enough once $n^{\frac1{10}}(1+o(1))<\widetilde{w}_n(k, x).$ 
Meanwhile, as for the upper bound of {\rm I}, we also need an additional lower bound for $\mathbb{P}(2n Y_{n-m_1(n)}\ge u_n(m_1(n), -\alpha_n)$ where $m_1(n)=[\sqrt{s_n}], \alpha_n=\frac{a(s_n)}{b(s_n)}.$ 

Hence, we present the following lemma.  
\begin{lem}\label{vi} Suppose that $\lim\limits_{n\to\infty} v=\infty.$ Given $j_n=o(n)$ and   recall $$\widetilde{w}_n(k, x)=\frac{k}{\sqrt{s_n}}+a(s_n)+b(s_n) x$$
for any $0\le k\le j_n.$  As $n$ large enough, we have four different estimates according to the range of $\widetilde{w}_n(k, x)$.  
\begin{itemize}
\item When $|\widetilde{w}_n(k, x)|\le M,$ 
$$\mathbb{P}(2n Y_{n-k}> u_n(x))\ge 1-\Phi(M+1). $$
\item If $ 1\ll \widetilde{w}_n(k, x)\le 2 n^{\frac{1}{10}},$
\begin{equation}\label{vie}
\mathbb{P}(2n Y_{n-k}>u_n(x))=\frac{1+O(\widetilde{w}_n^{-2}(k, x))}{\sqrt{2\pi} \widetilde{w}_n(k, x)} e^{-\frac{\widetilde{w}_n^2(k, x)}{2}}. \end{equation}
\item For the case $2n^{\frac{1}{10}}\sqrt{\log s_n}<x\le \sqrt{n\log s_n},$ 
\begin{equation}\label{viae1}\mathbb{P}(2n Y_{n}>u_n(x))\le 
2n^{-\frac{1}{10}}e^{-\frac{1}{12} \widetilde{w}_n^2(0, x)}.  \end{equation}
\item At last, when $x> \sqrt{n\log s_n},$
\begin{equation}\label{viae2}\mathbb{P}(2n Y_{n}>u_n(x))\le 
2n^{-1/2} e^{-\frac{1}{6} \sqrt{n} \widetilde{w}_n(0, x)}.
\end{equation}
 \end{itemize}
\end{lem}

Having this lemma at hands, the argument for Theorem \ref{main} when $v$ tends to infinity is ready. We give a quick review on the procedure.  
Indeed, \eqref{vie} taking the place of \eqref{clt*} with the same calculus guarantees that 
$$\aligned  
\int_{-\ell_1(n)}^{\ell_2(n)} |\mathbb{P}(2n Y_{(n)}\le u_n(x))-e^{e^{-x}}| dx&=\frac{(\log \log s_n)^2}{2\log s_n}(1+o(1)).
\endaligned $$ 
To verify the desired limit  
 $$\lim_{n\to\infty} \frac{(\log \log s_n)^2}{\log s_n}W_1(\mathcal{L}(X_n), \Lambda)=\frac12,$$ it suffices to prove that 
 \begin{equation}\label{finalineq}\rm{I}+\rm{I\!I\!I}\ll \frac{(\log\log s_n)^2}{\log s_n}.\end{equation}

We copy the upper bound on {\rm I}, which says 
\begin{equation} \label{boundI}\aligned {\rm I}
&\le e^{-\ell_1(n)-e^{\ell_1(n)}}+x_0\mathbb{P}^{m_1(n)}\big(2n Y_{n-m_1(n)}\le u_n\big(-a(s_n)/b(s_n)\big)\big)\\
&\quad \quad \quad \quad\quad+\log s_n\, \exp\big\{-\sum_{k=0}^{m_2(n)}\mathbb{P}(2n Y_{n-k}\ge u_n(-\ell_1(n)))\big\}. \endaligned  \end{equation}
Comparing with the jobs done in the precedent section, the first and the third terms in \eqref{boundI} have the same estimates as those related to the case when $v$ is bounded. The only thing to do is on the second one. 
The choice of $m_1(n)$ makes $\widetilde{w}_n(m_1(n), -\alpha_n)=1.$ Thus, the second assertion of Lemma \ref{vi} with $M=1,$ indicates   
$$\aligned \mathbb{P}\big(2n Y_{n-m_1(n)}>u_n(-\alpha_n))
&\ge 1-\Phi(2). \endaligned  $$ 
Hence $\mathbb{P}\big(2n Y_{n-m_1(n)}\le u_n(-\alpha_n)\big)\le \Phi(2),$
which guarantees that the second term in \eqref{boundI} is included by $o((\log\log s_n)^2/\log s_n).$ Thus 
$${\rm I}\ll\frac{\log\log s_n)^2}{\log s_n}.$$

To verify \eqref{finalineq}, it remains to confirm the last term 
$${\rm I\!I\!I}\ll \frac{\log\log s_n)^2}{\log s_n}.$$
Indeed, checking the treatment done on ${\rm I\!I\!I}$ in the third section, we have 
$$\aligned {\rm I\!I\!I}&\le (\log s_n)^{-1}+\int_{\ell_2(n)}^{\infty}\big(\sum_{k=0}^{j_n-1}\mathbb{P}(2nY_{n-k}\ge u_n(x))+(n-j_n) \mathbb{P}(2nY_{n-j_n}\ge u_n(x))\big)dx.
\endaligned $$ 
According to Lemma \ref{vi}, we cut the integral into three parts and 
use the simpler upper bound  
$$\sum_{k=0}^{j_n-1}\mathbb{P}(2nY_{n-k}\ge u_n(x))+(n-j_n) \mathbb{P}(2nY_{n-j_n}\ge u_n(x))\le n\mathbb{P}(2nY_n\ge u_n(x))$$
for the last two to get 
$$\aligned {\rm I\!I\!I}&\le (\log s_n)^{-1} +\int_{\ell_2(n)}^{2n^{\frac{1}{10}}\sqrt{\log s_n}}\left(n\mathbb{P}(2n Y_{n-j_n}>u_{n}(x)) +\sum_{k=0}^{j_n-1} \mathbb{P}(2n Y_{n-k}>u_{n}(x))\right)dx\\
&\quad+n\int_{2n^{\frac{1}{10}}\sqrt{\log s_n}}^{n^{1/2}\sqrt{\log s_n}}\mathbb{P}(2n Y_{n}>u_{n}(x))dx+n\int_{n^{1/2}\sqrt{\log s_n}}^{\infty}\mathbb{P}(2n Y_{n}>u_{n}(x))dx\\
&\le (\log s_n)^{-1} +\int_{\ell_2(n)}^{+\infty}\left(\frac{ne^{-\frac{\widetilde{w}_n^2(j_n, x)}{2}}}{\widetilde{w}_n(j_n, x)}+\sum_{k=0}^{j_n-1}\frac{e^{-\frac{\widetilde{w}_n^2(k, x)}{2}}}{\widetilde{w}_n(k, x)}\right)dx\\
&\quad+2n^{\frac 9{10}} \int_{2n^{\frac{1}{10}}\sqrt{\log s_n}}^{\infty}e^{-\frac1{12}\widetilde{w}_n^2(0, x)} dx+2n^{1/2}\int_{n^{1/2}\sqrt{\log s_n}}^{\infty}e^{-\frac{1}{6}\widetilde{w}_n(0, x) } dx.
\endaligned $$ 
Lemma \ref{wnkx} guarantees that  
$$\int_{\ell_2(n)}^{+\infty}\sum_{k=0}^{j_n-1}\frac{e^{-\frac{\widetilde{w}_n^2(k, x)}{2}}}{\widetilde{w}_n(k, x)}dx=O\left(\frac{\sqrt{s_n\log s_n}}{\widetilde{w}_n^3(0, \ell_2(n))}e^{-\frac12\widetilde{w}_n^2(0, \ell_2(n))}\right)=O(\frac{1}{\log s_n}),$$
whence, 
$${\rm I\!I\!I}\ll \frac{(\log\log s_n)^2}{\log s_n}.$$

 Thereby the proof is completed now. 

All left is to prove Lemma \ref{vi}. 

\begin{proof}[\bf Proof of Lemma \ref{vi}] 
We first do dome preparations on $\mathbb{P}(2n Y_{j}\ge t).$ 

The expression (4.1) in \cite{MW24} tells 
\begin{equation}\label{keyexm0} \mathbb{P}(2n Y_{j}\ge t)=\frac{2^{v+1/2}v^{2j+v-1/2}\Gamma(j+v+1/2)}{\,\Gamma(2j+2v)\,\Gamma(j)}\int_{t/v}^{+\infty}e^{-v\tau_j(y)} dy(1+O(v^{-1}))
\end{equation}
for any $t>0.$ Here, the function $\tau_j$ is defined as 
$$\tau_j(y)=\sqrt{1+y^2}-\log(1+\sqrt{1+y^2})+\frac{1}{4v}\log(1+y^2)-\frac{2j-1}{v}\log y, 
$$
which is strictly convex and whose first and second derivatives are $$\tau_j'(y)=\frac{y}{1+\sqrt{1+y^2}}+\frac{y}{2v(1+y^2)}-\frac{2j-1}{v \, y}$$ and 
$$\tau_j''(y)=\frac{1}{\sqrt{1+y^2}(1+\sqrt{1+y^2})}+\frac{1-y^2}{2 v(1+y^2)^2}+\frac{2j-1}{v y^2}.$$  
Denoted by $x_j$ the unique minimizer of $\tau_j,$ which is verified in \cite{MW24} satisfying 
$$2\sqrt{(j-1)(j+v-1)}\le v x_j\le 2\sqrt{j(j+v)}.$$ 
Although $x_j$ is the unique solution to $\tau_j'(x)=0,$ which is supposed to be the right point for $\tau_j$ to expand at (see \cite{JQ} for details on $x_j$ and $\tau_j''(x_j)$), limited by the lack of the analytical formula of $x_j$ we use $\mu_j:=\frac{2\sqrt{j(j+v)}}{v}$ to take the place of $x_j$ because $$1+\mu_j^2=\frac{(2j+v)^2}{v^2}.$$ 
Hence,  
$$\aligned v\tau_j\big(\mu_j)&=2j+v-v \log \frac{2(j+v)}{v}+\frac{1}{2}\log\frac{2j+v}{v}-(2j-1)\log\frac{2\sqrt{j(j+v)}}{v}\\
&=2j+v-(v+2j-1)\log 2+(v+2j-\frac32)\log v\\
&\quad -(v+j-\frac{1}{2})\log(j+v)-(j-\frac12)\log j
\endaligned $$
and 
$$\aligned \tau_j'\big(\mu_j\big)&=\frac{1}{2\sqrt{j(j+v)}}+\frac{\sqrt{j(j+v)}}{(2j+v)^2}.
\endaligned $$

We treat first the coefficient in \eqref{keyexm0}. Indeed, use the Stirling formula and the asymptotic  
$$\frac{(j+v+1/2)^{j+v}}{(j+v)^{j+v}}=e^{(j+v)\log(1+1/(2(j+v)))}=e^{1/2+O(n^{-1})}$$ to obtain 
$$\aligned \frac{2^{v+1/2}v^{2j+v-1/2}\Gamma(j+v+1/2)}{\,\Gamma(2j+2v)\,\Gamma(j)}
&=\frac{v^{2j+v-1/2}2^{v+1/2}(j+v+1/2)^{j+v}}{\sqrt{2\pi}(2j+2v)^{2j+2v-1/2} j^{j-1/2}}e^{2j+v-1/2}(1+O(n^{-1}))\\
&=\frac{v^{2j+v-1/2} e^{2j+v}}{\sqrt{2\pi} 2^{2j+v-1}(j+v)^{j+v-1/2} j^{j-1/2}} (1+O(n^{-1})),
\endaligned $$
which is recognized to be $$\frac{(1+O(n^{-1}))v}{\sqrt{2j+v}\sqrt{2\pi}}\exp\{v\tau_j\big(\mu_j\big)\}.$$ 
Thus, we rewrite \eqref{keyexm0} as 
$$\mathbb{P}(2n Y_{j}\ge t)=\frac{v(1+\tilde{O}(n^{-1}))}{\sqrt{2j+v}\sqrt{2\pi}}\int_{t/v}^{+\infty}e^{-v(\tau_j(y)-\tau_j(\mu_j))} dy,
$$  
whence it follows from the substitution $y=\mu_j+\frac{\sqrt{2j+v}}{v} z$ that 
$$\aligned\mathbb{P}(2n Y_{j}\ge u_n(x))&=\frac{1+\tilde{O}(n^{-1})}{\sqrt{2\pi}}\int_{\frac{u_n(x)-v\mu_j}{\sqrt{2j+v}}}^{\infty} e^{-v(\tau_j(\mu_j+\frac{\sqrt{2j+v}}{v} z))-\tau_j(\mu_j))} dz.
\endaligned $$
As far as the term $\frac{u_n(x)-v\mu_j }{\sqrt{2j+v}}$ for $n-j_n\le j\le n$ is concerned, with  $k=n-j,$ it is easy to see that 
$$\aligned &\quad \frac{u_n(x)-2\sqrt{(n-k)(n-k+v)}}{\sqrt{2(n-k)+v}}\\
 &=\frac{2\sqrt{n(n+v)}-2\sqrt{(n-k)(n-k+v)}}{\sqrt{2(n-k)+v}}+\frac{\sqrt{2n+v}}{\sqrt{2(n-k)+v}}(a(s_n)+b(s_n)\, x)\\
 &=\widetilde{w}_n(k, x)(1+O(\frac k n)).
\endaligned $$
This implies in further that 
\begin{equation}\label{exp2}\aligned\mathbb{P}(2n Y_{j}\ge u_n(x))&=\frac{1+\tilde{O}(n^{-1})}{\sqrt{2\pi}}\int_{\widetilde{w}_n(k, x)(1+O(k/n))}^{\infty} e^{-v(\tau_j(\mu_j+\frac{\sqrt{2j+v}}{v} z))-\tau_j(\mu_j))} dz.
\endaligned \end{equation}

Now, we dig the integrand in \eqref{exp2}. In fact, apply Taylor's formula to see 
\begin{equation}\label{sumtay}\aligned v\tau_j(\mu_j+\frac{\sqrt{2j+v}}{v} z)-v\tau_j(\mu_j)&=z\sqrt{2j+v }\, \tau_j'(\mu_j)+\frac{z^2(2j+v) }{2 v} \tau_j''(\mu_j+\frac{\sqrt{2j+v}}{v}\theta)
\endaligned \end{equation}
for some $\theta\in (0, z).$
Considering the expression of $\tau_j'(\mu_j),$ we know  
$$z\sqrt{2j+v }\tau_j'(\mu_j)=O(|z|(j^{-1/2}+j^{1/2}(2j+v)^{-1}))=O(|z|n^{-1/2}).$$
For $\tau_j''(\mu_j+\frac{\sqrt{2j+v}}{v} \theta),$  we need to simplify first the form of $$1+(\mu_j+\frac{\sqrt{2j+v}}{v} \theta)^2.$$ Now $ \mu_j=2\sqrt{j(j+v)}/v,$ it follows  
$$\aligned 1+(\mu_j+\frac{\sqrt{2j+v}}{v} \theta)^2&=1+\mu_j^2+\frac{2\mu_j\sqrt{2j+v}}{v}\theta+\frac{2j+v}{v^2}\theta^2\\
&=\frac{(2j+v)^2}{v^2}\left(1+\frac{4\sqrt{j(j+v)}\theta }{(2j+v)^{3/2}}+\frac{\theta^2}{2j+v}\right)\\
&=\frac{(2j+v)^2}{v^2}(1+O(\frac{\theta^2}{2n+v}))
\endaligned $$
for any $0<\theta\ll n^{1/2}, $ whence  
$$\sqrt{1+(\mu_j+\frac{\sqrt{2j+v}}{v} \theta)^2}=\frac{2j+v}{v}(1+O(\frac{\theta^2}{2n+v})).$$ 
Thereby, it follows from some simple calculus that 
$$\aligned\tau_j''(\mu_j+\frac{\sqrt{2j+v}}{v} \theta)&=\frac{1}{\sqrt{1+y^2}(1+\sqrt{1+y^2})}+\frac{1-y^2}{2 v(1+y^2)^2}+\frac{2j-1}{v y^2}\\
&=(\frac{v^2}{2(2j+v)(j+v)}+\frac{v}{2(j+v)})\big(1+O\big(\frac{\theta^2}{2n+v}\big)\big)\\
&=\frac{v}{2j+v}\big(1+O\big(\frac{\theta^2}{2n+v}\big)\big). \endaligned$$ 
Thus, we have 
$$\aligned v\tau_j(\mu_j+\frac{\sqrt{2j+v}}{v} z)-v\tau_j(\mu_j)&=\frac{z^2}{2}\big(1+O\big(\frac{z^2}{2n+v}\big)\big)+O(|z|n^{-1/2})
\endaligned $$ 
and then 
\begin{equation}\label{expre*} v\tau_j(\mu_j+\frac{\sqrt{2j+v}}{v} z)-v\tau_j(\mu_j)=\frac{z^2}{2}+o(1)
\end{equation} 
once $|z|\ll n^{1/4}.$ 
Now $\tau_j''(y)\ge \frac{j}{vy^2},$ which with Taylor's formula, implies 
$$\aligned \tau_j(\mu_j+\frac{\sqrt{2j+v}}{v} z)-\tau_j(\mu_j)&=\frac{2j+v}{v^2}\int_0^z\int_0^{t}\tau_j''(\mu_j+\frac{\sqrt{2j+v}}{v} s) ds dt\\
&\ge \frac{(2j+v)j}{v^3} \int_0^z\int_0^{t}(\mu_j+\frac{\sqrt{2j+v}}{v} s)^{-2} ds dt\\
&=\frac{j}{v}\big(\frac{\sqrt{2j+v} z}{2\sqrt{j(j+v)}}-\log\big(1+\frac{\sqrt{2j+v} z}{2\sqrt{j(j+v)}}\big ) \big).
\endaligned $$
Setting for simplicity $\beta_j=\frac{\sqrt{2j+v}}{2\sqrt{j(j+v)}},$ one obtains 
$$\aligned v\tau_j(\mu_j+\frac{\sqrt{2j+v}}{v} z)-v\tau_j(\mu_j)&\ge j\big(\beta_j z-\log\big(1+\beta_j z\big ) \big),
\endaligned $$
whence 
$$\aligned\mathbb{P}(2n Y_{j}\ge u_n(x))&\le\frac{1+o(1)}{\sqrt{2\pi}}\int_{\widetilde{w}_n(k, x)(1+O(kn^{-1}))}^{+\infty} e^{-j\beta_j(z-\beta_j^{-1}\log(1+\beta_j z)) } dz.
\endaligned$$ 
Let $h(z)=z-\beta_j^{-1}\log(1+\beta_j z),$ whose derivative satisfies 
$$h'(z)=1-\frac{1}{1+\beta_j z}=\frac{\beta_j z}{1+\beta_j z}\ge \frac{\beta_j a}{1+\beta_j a}$$ 
for any $z\ge a>0$ and then  
$$\int_a^{\infty} e^{-j\beta_j h(z)}dz\le\frac{1+\beta_j a}{j\beta_j^2 a}\int_a^{\infty}e^{-j\beta_j h(z)}h'(z)dz=\frac{1+\beta_j a}{j\beta_j^2 a} e^{-j\beta_j h(a)}.$$ 
Here, we use the fact $\lim_{z\to+\infty}e^{-j\beta_j h(z)}=0.$ 
As a natural consequence, with $a=\widetilde{w}_n(k, x)(1+O(\frac k n))$ and the fact $1/4\le j \beta_j^2\le 1/2,$ it follows from \eqref{exp2} that 
\begin{equation}\label{forj}\mathbb{P}(2n Y_{j}\ge u_n(x))\le\frac{(4+o(1))(1+\beta_j a)}{\sqrt{2\pi} a } e^{-j(\beta_j a-\log(1+\beta_j a))}.
\end{equation}
Now the function $y-\log(1+y)$ is increasing on $y>0$ and 
$$\frac{\widetilde{w}_n(k, x)}{3\sqrt{n}}\le \frac{\widetilde{w}_n(k, x)(1+O(\frac k n))}{2\sqrt{j}}\le \beta_j a\le \frac{\widetilde{w}_n(k, x)(1+O(\frac k n))}{\sqrt{2j}}\le \frac{\widetilde{w}_n(k, x)}{\sqrt{n}},$$ whence 
 \begin{equation}\label{ups}\aligned\mathbb{P}(2n Y_{j}\ge u_n(x))&\le\frac{(4+o(1))(\sqrt{n}+ \widetilde{w}_n(k, x))}{\sqrt{2\pi} \sqrt{n}\, \widetilde{w}_n(k, x)} e^{-j(\frac{\widetilde{w}_n(k, x)}{3\sqrt{n}}-\log(1+\frac{\widetilde{w}_n(k, x)}{3\sqrt{n}})) }\\
 &\le\frac{4}{\sqrt{n}\wedge \widetilde{w}_n(k, x)} e^{-j(\frac{\widetilde{w}_n(k, x)}{3\sqrt{n}}-\log(1+\frac{\widetilde{w}_n(k, x)}{3\sqrt{n}})) }
\endaligned\end{equation} 
for all $n-j_n\le j\le n.$
As a parallel result, we also have 
\begin{equation}\label{ineq2}
	\int_{a}^{\infty} e^{-v(\tau_j(\mu_j+\frac{\sqrt{2j+v}}{v} z))-\tau_j(\mu_j))} dz\le \frac{4(\sqrt{n}+ a)}{\sqrt{n} a} e^{-j(\frac{a}{3\sqrt{n}}-\log(1+\frac{a}{3\sqrt{n}}))}
\end{equation}   
for any $a>0.$
Next, we are going to verify the equality \eqref{vie}.  
Indeed, plugging \eqref{expre*} and \eqref{ups} with $a=n^{1/5}$ into \eqref{exp2}, for any $x$ such that $\widetilde{w}_n(k, x)\le n^{1/5},$
it holds 
\begin{equation}\label{sum00}\aligned &\quad \mathbb{P}(2n Y_{j}\ge u_n(x))\\
&=\frac{1}{\sqrt{2\pi}}\int_{\widetilde{w}_n(k, x)(1+O(k/n))}^{n^{1/5}}e^{-\frac{z^2}{2}} dz (1+\tilde{O}(n^{-1}+z^4 (n+v)^{-1}+z n^{-1/2}))\\
&\quad+\frac{1+\tilde{O}(n^{-1})}{\sqrt{2\pi}}\int_{n^{1/5}}^{+\infty} e^{-v(\tau_j(\mu_j+\frac{\sqrt{2j+v}}{v} z))-\tau_j(\mu_j))} dz.
\endaligned \end{equation}  
The second integral in \eqref{sum00}, via \eqref{ineq2}, has the same order as  
$$
n^{-1/5} e^{-j(\frac{1}{3}n^{-\frac3{10}}-\log(1+\frac{1}{3}n^{-\frac3{10}}))}.$$ 
For the first integral in \eqref{sum00}, we see 
$$\aligned &\quad\int_{\widetilde{w}_n(k, x)(1+O(k/n))}^{n^{1/5}}e^{-\frac{z^2}{2}} dz (1+\tilde{O}(n^{-1}+z^4 (n+v)^{-1}+z n^{-1/2}))\\
&=(1+\tilde{O}(n^{-1/5}))\big(\int_{\widetilde{w}_n(k, x)(1+O(k/n))}^{+\infty}e^{-\frac{z^2}{2}} dz-\int_{1/5}^{\infty} e^{-\frac12 z^2} dz\big)\\
&=(1+\tilde{O}(n^{-1/5}+\frac kn+\widetilde{w}_n^{-2}(k, x)))\widetilde{w}_n^{-1}(k, x)e^{-\frac{\widetilde{w}_n^2(k, x)}{2}(1+O(\frac kn))}+O(n^{-\frac15}e^{-\frac{1}{2}n^{\frac25}})\\
&=(1+\tilde{O}(n^{-1/5}+\widetilde{w}_n^{2}(k, x)\frac kn+\widetilde{w}_n^{-2}(k, x)))\widetilde{w}_n^{-1}(k, x)e^{-\frac{\widetilde{w}_n^2(k, x)}{2}}\\
&=(1+O(\widetilde{w}_n^{-2}(k, x)))\widetilde{w}_n^{-1}(k, x)e^{-\frac{\widetilde{w}_n^2(k, x)}{2}}.
\endaligned$$
Now, we need to compare these two orders. Considering 
the facts $x-\log(1+x)\ge \frac14 x^2$ for $|x|$ small enough, and $$j\ge n-j_n\ge \frac{4 n}{5}$$ and the constrain  $\widetilde{w}_n(t, x)\le 2n^{\frac{1}{10}},$ one knows  
$$\mathbb{P}(2n Y_{j}\ge u_n(x))=\frac{1+O(\widetilde{w}_n^{-2}(k, x))}{\sqrt{2\pi}\widetilde{w}_n(k, x)}e^{-\frac{\widetilde{w}_n^2(k, x)}{2}},$$
which is exactly the wanted \eqref{vie}.  
For \eqref{viae1} and \eqref{viae2}, the upper bound \eqref{forj} is applied with $j=n$, we see 
$$\aligned\mathbb{P}(2n Y_{n}\ge u_n(x))&\le\frac{(4+o(1))(1+\beta_n \widetilde{w}_n(0, x))}{\sqrt{2\pi} \widetilde{w}_n(0, x) } e^{-n(\beta_n \widetilde{w}_n(0, x)-\log(1+\beta_n \widetilde{w}_n(0, x)))}
\endaligned$$
and then similarly it follows from $(2n^{-1/2})\le \beta_n \le (2n)^{-1/2}$ that 
\begin{equation}\label{ups1}\aligned\mathbb{P}(2n Y_{n}\ge u_n(x))&\le\frac{(4+o(1))(1+\beta_n \widetilde{w}_n(0, x))}{\sqrt{2\pi} \widetilde{w}_n(0, x) } e^{-n(\frac{ \widetilde{w}_n(0, x)}{2\sqrt{n}}-\log(1+\frac{ \widetilde{w}_n(0, x)}{2\sqrt{n}}))}
\endaligned\end{equation}

When $2n^{\frac{1}{10}}<\widetilde{w}_n(0, x)\le \sqrt{n},$ which implies $$ 0<\frac{\widetilde{w}_n(0, x)}{2\sqrt{n}}\le \frac{1}{2}.$$  The elementary inequality 
$$x-\log(1+x)\ge \frac{x^2}{3} \quad \text{for all} \quad x\in (0, 1/2], $$ with \eqref{ups1} gains  
$$\mathbb{P}(2n Y_{n}\ge u_n(x))\le 2n^{-\frac{1}{10}}e^{-\frac{\widetilde{w}_n^2(0, x)}{12 }}.
$$
Similarly, when $\widetilde{w}_n(0, x)>n^{1/2},$ we use instead another lower bound  $$x-\log (1+x)\ge \frac{x}{3} \quad \text{for all} \quad x\ge 1/2$$ to get 
$$\mathbb{P}(2n Y_{n}\ge u_n(x))\le 2 n^{-1/2}e^{-\frac{\sqrt{n}\widetilde{w}_n(0, x)}{6}}.
$$
At last, for $|\widetilde{w}_n(k, x)|\le M,$ \eqref{exp2}  brings us 
$$\aligned \mathbb{P}(2n_{n-k}\ge u_n(x))&\ge\frac{1+o(1)}{\sqrt{2\pi}}\int_{M+o(1)}^{+\infty}e^{-v(\tau_j(\mu_j+\frac{\sqrt{2j+v}}{v} z))-\tau_j(\mu_j))} dz\\
&\ge\frac{1+o(1)}{\sqrt{2\pi}}\int_{M+o(1)}^{n^{1/6}\sqrt{\log s_n}} e^{-\frac{z^2}{2}} dz\\
&=1-\Phi(M+o(1))-\frac{1+o(1)}{\sqrt{2\pi}}\int_{n^{1/6}\sqrt{\log s_n}}^{+\infty} e^{-\frac{z^2}{2}} dz\\
&=1-\Phi(M)+o(1),
\endaligned $$
whence 
$$\mathbb{P}(2n_{n-k}\ge u_n(x))\le \Phi(M+1)$$ uniformly on $k$ and $x$ such that 
$|\widetilde{w}_n(k, x)|\le M.$ The proof of Lemma \ref{vi} is then finished. 
 \end{proof} 

\section{Proof of Theorem \ref{main2}} 
Examining the proof of Theorem \ref{main}, we see 
$$\aligned\sup_{x\in\mathbb{R}}|F_n(x)-e^{-e^{-x}}|&=\sup_{x\in (-\ell_1(n), \ell_2(n))}|F_n(x)-e^{-e^{-x}}|\\
&=\frac{(1+o(1))}{2\log s_n}\sup_{x\in (-\ell_1(n), \ell_2(n))}e^{-e^{-x}-x}(\ell_2(n)-x)^2.
\endaligned $$
Set $$g(x)=e^{-x}+x-2\log  (\ell_2(n)-x), \quad -\ell_1(n)\le x\le \ell_2(n)$$
and we are going to identify the minimizer of $g.$  
Now $$g'(x)=1-e^{-x}+\frac{2}{\ell_2(n)-x},$$ which is strictly increasing and $g'(-\ell_1(n))<0$ and $g'(\ell_2(n))=+\infty.$ Thus, $g$ has a unique minimizer, denoted by $x_0.$ 
It is ready to know $g'(0)>0$ and 
$$\aligned g'(-\frac{2}{\ell_2(n)})&=1-e^{\frac{2}{\ell_2(n)}}+\frac{2}{\ell_2(n)+\frac{2}{\ell_2(n)}}\\
&=-\frac{2}{\ell_2(n)}-\frac{2}{\ell_2^2(n)}+\frac{2}{\ell_2(n)+\frac{2}{\ell_2(n)}}+o(\ell_2^{-2}(n))\\
&<0,
\endaligned $$ 
whence $$-\frac{2}{\ell_2(n)}<x_0<0.$$ 
As a consequence 
$$g(x_0)<g(0)=1-2\log \ell_2(n)$$ and also it follows from the monotonicity of involved functions that 
$$g(x_0)>1-\frac{2}{\ell_2(n)}-2\log\big(\ell_2(n)+\frac{2}{\ell_2(n)}\big)=1-2\log\ell_2(n)+o(1).$$
Therefore, 
$$\aligned \sup_{x\in (-\ell_1(n), \ell_2(n))}e^{-e^{-x}-x}(\ell_2(n)-x)^2=e^{-g(x_0)}=e^{-1} \ell_2^2(n)(1+o(1)),
\endaligned $$
which implies 
$$\sup_{x\in\mathbb{R}}|F_n(x)-e^{-e^{-x}}|=\frac{(\log\log s_n)^2(1+o(1))}{2 e \log s_n}.$$ 
The proof is then finished. 
\subsection*{Conflict of interest}  The authors have no conflicts to disclose.

\subsection*{Acknowledgment} The authors would like to thank their colleague Professor Xin-xin Chen for help discussions during lunch time.


\begin{thebibliography}{SOSL90}
\bibitem{AS} 
M. Abramowitz and I. A. Stegun. \emph{Handbook of Mathematical Functions}, 2nd ed., Dover Publications, New York, 1972.

\bibitem{Akemann} 
G. Akemann. The complex Laguerre symplectic ensemble of non-Hermitian matrices. \emph{Nuclear Phys. B}, \textbf{730}(2005), 253-299.

\bibitem{ABDbook} 
G. Akemann, J. Baik, and P. Di Francesco. \emph{The Oxford Handbook of Random Matrix Theory}. Oxford University Press, Oxford, 2011.

\bibitem{AB} 
G. Akemann and M. Bender. Interpolation between Airy and Poisson statistics for unitary chiral non-Hermitian random matrix theory. \emph{J. Math. Phys.}, \textbf{51}(2010), 103524.

\bibitem{ABK} 
G. Akemann, S.-S. Byun, and N.-G. Kang, A non-Hermitian generalization of the Marchenko-Pastur distribution: from the circular law to multi-criticality, \emph{Ann. Henri. Poincar\'e}, \textbf{22}(2021), 1035-1068.
\bibitem{AnG} 
G. Anderson, A. Guionnet, and O. Zeitouni. \emph{An Introduction to Random Matrices}. Cambridge University Press, Cambridge, 2010.

\bibitem{BSbook} 
Z. D. Bai and J. Silverstein. \emph{Spectral Analysis of Large Dimensional Random Matrices}. 2nd ed., Springer, New York, 2009.

\bibitem{BDSbook} 
J. Baik, P. Deift, and P. Suidan. \emph{Combinatorics and Random Matrix Theory}. American Mathematical Society, USA, 2016.

\bibitem{Bender}
M. Bender. Edge scaling limits for a family of non-Hermitian random matrix ensembles. \emph{Probab. Th. Relat. Fields}, \textbf{147}(2010), 241-271.
\bibitem{BC12} C. Bordenave and D. Chafa$\ddot{\rm \imath}$. Around the circular law. \emph{Probab. Surv.}, \textbf{9}(2012), 1-89.
\bibitem{BCC18} C. Bordenave, P. Caputo, D. Chafa$\ddot{\rm \imath}$, and K. Tikhomirov. On the spectral radius of a random matrix: An upper bound without fourth moment. \emph{Ann. Probab.}, \textbf{46}(2018), 2268-2286.
\bibitem{BCG22} C. Bordenave, D. Chafa$\ddot{\rm \imath}$, and D. Garcia-Zelada. Convergence of the spectral radius of a random matrix through its characteristic polynomial. \emph{Probab. Th. Relat. Fields}, \textbf{182}(2022), 1163-1181.

 \bibitem{Byun 2023}
		S.-S. Byun and P. J. Forrester. \emph{Progress on the Study of the Ginibre Ensembles II: GinOE and GinSE}. arXiv.2301.05022.
\bibitem{Byun 2025}
 S.-S. Byun and P. J. Forrester. \emph{Progress on the Study of the Ginibre Ensembles}. 1st ed., Springer Singapore, 2025. 
\bibitem{Chafaii18} D. Chafa$\ddot{\rm \imath}$. Around the circular law: an update. (Webblog)
https://djalil.chafai.net/blog/2018/11/04/around-the-circular-law-an-update.

\bibitem{Chafaii} D. Chafa$\ddot{\rm \imath}$ and S. P\'ech\'e. A note on the second order universality at the edge of Coulomb gases on the plane. \emph{J. Stat. Phys.}, \textbf{156}(2014), 368-383.
 \bibitem{JQ} S. Chang, T. Jiang, and Y. Qi.  Eigenvalues of large chiral non-Hermitian random matrices. \emph{J. Math. Phys.}, \textbf{61}(2020), 013508.
 \bibitem{Charlier03} C. Charlier. Large gap asymptotics on annuli in the random normal matrix model. \emph{Math. Ann.}, \textbf{388}(2024), 3529-3587.
 \bibitem{Charlier22} C. Charlier. Asymptotics of determinants with a rotation-invariant weight and discontinuities along circles. \emph{Adv. Math.}, \textbf{408}(2022), 108600.

\bibitem{CipoErXu} G. Cipolloni, L. Erd$\ddot{o}$s and Y. Xu. Universality of extremal eigenvalues of large random matrices. arXiv:2312.08325. 

\bibitem{DiF} 
P. Di Francesco, M. Gaudin, C. Itzykson, and F. Lesage. Laughlin's wave functions, Coulomb gases and expansions of the discriminant. \emph{Int. J. Mod. Phys. A}, \textbf{9}(1994), 4257-4351.

\bibitem{Fenzl}
M. Fenzl and G. Lambert. Precise deviations for disk counting statistics of invariant determinantal processes. \emph{Int. Math. Res. Not.}, \textbf{2022}(10)(2022), 7420-7494.

\bibitem{Forrester} 
P. J. Forrester. \emph{Log-gases and Random Matrices}, London Mathematical Society Monographs Series \textbf{34}, Princeton University Press, Princeton, NJ, 2010.

\bibitem{Grobe} 
R. Grobe, F. Haake, and H.-J. Sommers. Quantum distinction of regular and chaotic dissipative motion. \emph{Phys. Rev. Lett.},  \textbf{61}(1988), 1899-1902.

\bibitem{Johansson}
K. Johannsson. Shape fluctuations and random matrices. \emph{Comm. Math. Phys.}, \textbf{209}(2000), 437-476.

\bibitem{JohnstoneMa} I. M., Johnstone and Z. Ma.  Fast approach to the Tracy-Widom law at the edge of GOE and GUE. \emph{Ann. Appl. Prob.},
\textbf{22}(5)(2012), 1962-1988. 

\bibitem{Karoui}  N. Karoui.  A rate of convergence result for the largest eigenvalue of complex white Wishart matrices. \emph{Ann. Probab.}, \textbf{34}(6)(2006), 2077-2117 . 
\bibitem{KSbook}
B. A. Khoruzhenko and H. J. Sommers. Non-Hermitian Random Matrix Ensembles, in: \emph{Oxford Handbook of Random Matrix Theory}. Oxford University Press, Oxford, 2011.

 \bibitem{Kostlan} 
 E. Kostlan. On the spectra of Gaussian matrices. \emph{Lin. Alg. Appl.}, \textbf{162}(1992), 385-388.

\bibitem{Lacroix}
B. Lacroix-A-Chez-Toine, A. Grabsch, S. N. Majumdar, and G. Schehr. Extremes of 2D Coulomb gas: universal intermediate deviation regime. \emph{J. Stat. Mech.}(2018), 013203.

\bibitem{LACT19} 
B. Lacroix-A-Chez-Toine, J. Garzon, A. M. Calva, S. H. Castillo, I. P. Kundu, S. N. Majumdar, and G. Schehr. Intermediate deviation regime for the full eigenvalue statistics in the complex Ginibre ensemble. \emph{Phys. Rev. E}, \textbf{100}(2019), 012137.

\bibitem{LACT2019}
B. Lacroix-A-Chez-Toine, S.N. Majumdar, and G. Schehr. Rotating trapped fermions in two dimensions and the complex Ginibre ensemble: Exact results for the entanglement entropy and number variance. \emph{Phys. Rev. A}, \textbf{99}(2019), 021602.

\bibitem{Stein} L. H. Y. Chen, L. Goldstein and Q. Shao. Normal approximation by Stein's method. Springer-Verlag Berlin Heidelberg 2011. 

\bibitem{L1WGinibre} Y. T. Ma and X. Meng. Berry-Esseen bound and $W1$ distance for Gumbel and the spectral radius of the complex Ginibre ensembles. http://arxiv.org/abs/2501.08039.  
\bibitem{MW24} Y.T. Ma and S. Wang. Deviation probabilities for extremal eigenvalues of large chiral non-Hermitian random matrices. \emph{Forum Math.}, https://doi.org/10.1515/forum-2023-0253. 
\bibitem{MP}
V. A. Marchenko and L. A. Pastur. Distributions of some sets of random matrices. \emph{Math. USSR-Sb}, \textbf{1}(1967), 457-483.
\bibitem{Olver} 
F. W. J. Olver, D. W. Lozier, R. F. Boisvert, and C. W. Clark. \emph{NIST Handbook Mathematical Functions}. Cambridge University Press, Cambridge, 2010.

\bibitem{osborn} 
J. C. Osborn. Universal results from an alternate random matrix model for QCD with a baryon chemical potential. \emph{Phys. Rev. Lett.}, \textbf{93}(2004), 2220019.
\bibitem{Petrov} V. V. Petrov. Sums of independent random variables. Springer-Verlag Berlin Heidelberg New York 1975.
\bibitem{Rider03} 
B.C. Rider. A limit theorem at the edge of a non-Hermitian random matrix ensemble. \emph{J. Phys. A}, \textbf{36}(2003), 3401-3409.
\bibitem{Rider04} B.C. Rider. Order statistics and Ginibre' s ensembles, \emph{ J. Stat. Phys.}, \textbf{114}(2004), 1139-1148.

\bibitem{Rider 14} B.C. Rider and C.D. Sinclair. Extremal laws for the real Ginibre ensemble. \emph{ Ann. Appl. Probab.}, \textbf{24}(4)(2014), 1621-1651.
\bibitem{SX23} K, Schnelli and Y. Xu. Quantitative Tracy-Widom laws for the largest eigenvalue of generalized Wigner matrices. \emph{Electron. J. Probab.}, \textbf{28}(2023), 1-38.  
\bibitem{Stephanov} M. A. Stephanov. Random matrix model for qcd at finite density and the nature of the quenched limit.\emph{ Phys. Rev. Lett.}, \textbf{76}(1996), 4472-4475. 
\bibitem{Villani} C. Villani. \emph{Optimal transport: old and new.} Springer-Verlag Berlin Heidelberg, 2009.
\end{thebibliography}
\end{document}